\documentclass[11pt]{article}

\usepackage{amsmath,amsfonts,amsthm,amssymb}
\usepackage[margin=1in,letterpaper]{geometry}
\usepackage{kbordermatrix}
\usepackage{physics}
\usepackage{graphicx}
\usepackage{tikz}
\usepackage{amsthm}
\theoremstyle{definition}
\newtheorem{defi}{Definition}[section]
\newtheorem{theorem}[defi]{Theorem}
\newtheorem{corollary}[defi]{Corollary}
\newtheorem{lemma}[defi]{Lemma}
\newtheorem{prop}[defi]{Proposition}
\newtheorem{conj}[defi]{Conjecture}
\newtheorem{claim}[defi]{Claim}
\newtheorem{example}[defi]{Example}
\newtheorem{remark}[defi]{Remark}

\usepackage{bm}

\DeclareMathOperator{\Coker}{Coker}

\DeclareMathOperator{\supp}{supp}

\DeclareMathOperator{\lk}{lk}
\DeclareMathOperator{\st}{st}

\DeclareMathOperator{\conv}{conv}

\title{Rigidity of Balanced Minimal Cycle Complexes}
\author{Ryoshun Oba\thanks{Department of Mathematical Informatics, Graduate School of Information Science and Technology, University of Tokyo, 7-3-1 Hongo, Bunkyo-ku, 113-8656, Tokyo Japan. Email: \texttt{ryoshun\_oba@mist.i.u-tokyo.ac.jp}}}

\begin{document}
\maketitle
\begin{abstract}
A $(d-1)$-dimensional simplicial complex $\Delta$ is balanced if its graph $G(\Delta)$ is $d$-colorable.
Klee and Novik obtained the balanced lower bound theorem for balanced normal $(d-1)$-pseudomanifolds $\Delta$ with $d\geq3$ by showing that the subgraph of $G(\Delta)$ induced by the vertices colored in $T$ is rigid in $\mathbb{R}^3$ for any $3$ colors $T$.
We show that the same rigidity result, and thus the balanced lower bound theorem, holds for balanced minimal $(d-1)$-cycle complexes with $d \geq 3$.
Motivated by the Stanley's work on a colored system of parameters for the Stanley-Reisner ring of balanced simplicial complexes, we further investigate the infinitesimal rigidity of non-generic realization of balanced, and more broadly $\bm{a}$-balanced, simplicial complexes.
Among other results, we show that for $d \geq 4$, a balanced homology $(d-1)$-manifold can be realized as an infinitesimally rigid framework in $\mathbb{R}^d$ such that each vertex of color $i$ lies on the $i$th coordinate axis. 
\end{abstract}
\section{Introduction}
Barnett's lower bound theorem~\cite{Bar} asserts that the boundary complex $\Delta$ of a simplicial $d$-polytope satisfies the inequality 
\begin{equation} \label{eq:LBT}
f_1(\Delta) \geq df_0(\Delta) -\binom{d+1}{2},
\end{equation}
where $f_i(\Delta)$ denotes the number of $i$-dimensional faces of $\Delta$.
Subsequently this was generalized to pseudomanifolds~\cite{Fog,Kal}.
Several variants of the lower bound theorem have been investigated~(e.g. \cite{GKN,St1}). In this paper we discuss the balanced lower bound theorem due to Goff, Klee, Novik~\cite{GKN} and Klee and Novik~\cite{KN}.
A $(d-1)$-dimensional simplicial complex $\Delta$ is \emph{balanced} (or \emph{completely balanced}) if its graph $G(\Delta)$ is $d$-colorable. The balanced lower bound theorem~\cite{KN} asserts that, for $d \geq 3$, every balanced normal $(d-1)$-pseudomanifold $\Delta$ satisfies the inequality 
\begin{equation} \label{eq:balLBT}
2h_2(\Delta) \geq (d-1)h_1(\Delta), 
\end{equation}
where $h_1(\Delta)=f_0(\Delta)-d$, $h_2(\Delta)=f_1(\Delta)-(d-1)f_0(\Delta)+\binom{d}{2}$.

A connection between the lower bound theorem and rigidity theory was pointed out by Kalai~\cite{Kal} (see also~\cite{Gro}). 
Kalai~\cite{Kal} noted that the inequality (\ref{eq:LBT}) follows immediately from the (generic) rigidity of the graph of $\Delta$ in $\mathbb{R}^d$. 
Goff, Klee, Novik~\cite{GKN} pointed out that the inequality (\ref{eq:balLBT}) holds if $G(\Delta)[\kappa^{-1}(T)]$ is rigid in $\mathbb{R}^3$ for every $3$-set $T \subseteq [d]:=\{1,\ldots,d\}$, where $G[W]$ denotes the subgraph of $G$ induced by $W$ and $\kappa$ is the proper $d$-coloring of $G(\Delta)$.
Klee and Novik~\cite{KN} showed this rigidity result, and thus the balanced lower bound theorem, for balanced normal $(d-1)$-pseudomanifolds for $d\geq 3$.
This rigidity result for balanced normal pseudomanifolds~\cite[Lemma 3.5]{KN} is based on the rigidity theorem of pseudomanifolds by Fogelsanger~\cite{Fog}. Fogelsanger introduced a superclass of pseudomanifolds, called minimal cycle complexes, and show that minimal cycle complexes admit a decomposition into minimal cycle complexes in such a way that the decomposition behaves nicely with respect to edge contractions. He used this decomposition together with vertex splitting~\cite{Whi} and gluing to show that the graph of any minimal $(d-1)$-cycle complex is rigid in $\mathbb{R}^d$ for $d\geq3$. The remarkable aspect of this proof of rigidity is that the inductive proof works within the same dimension, so it is applicable to a class of simplicial complexes which is not closed under taking links. 
Fogelsanger's idea is recently used to show the global rigidity of pseudomanifolds~\cite{CJT} and $\mathbb{Z}_2$-symmetric rigidity of $\mathbb{Z}_2$-symmetric pseudomanifolds~\cite{CJT2}.

In this paper, we give a new application of Fogelsanger's idea to prove rigidity results of balanced simplicial complexes. We generalize the rigidity result~\cite[Lemma 3.5]{KN} of balanced normal pseudomanifolds to balanced minimal cycle complexes (Theorem~\ref{thm:rank-selcted}). Using an argument by Goff, Klee, Novik~\cite{GKN}, this immediately implies the balanced lower bound theorem for minimal cycle complexes (Corollary~\ref{cor:balanced_LBT}).

Balanced simplicial complexes, and more broadly $\bm{a}$-balanced simplicial complexes, have also been studied in the theory of Stanley-Reisner ring~\cite{Sta}. 
For $\bm{a}=(a_1,\ldots,a_m) \in \mathbb{Z}_{>0}^m$ with $\sum_{i=1}^m a_i=d$, a $(d-1)$-dimensional simplicial complex $\Delta$ on the vertex set $V(\Delta)$ is \emph{$\bm{a}$-balanced} if there is a map $\kappa:V(\Delta) \rightarrow [m]$ satisfying $|F \cap \kappa^{-1}(i)|\leq a_i$ for any $F \in \Delta$ and $i \in [m]$.
We call such a map $\kappa$ an \emph{$\bm{a}$-coloring} of $\Delta$.
Given a simplicial complex $\Delta$, the \emph{Stanley-Reisner ring} of $\Delta$ is $\mathbb{R}[\Delta]:=\mathbb{R}[x_v:v \in V(\Delta)]/I_\Delta$, where $I_\Delta$ is the ideal generated by monomials $\prod_{v \in G} x_v$ over all $G \not\in \Delta$.
For $\bm{a} \in \mathbb{Z}_{>0}^m$ and an $\bm{a}$-balanced simplicial complex $\Delta$ with an $\bm{a}$-coloring $\kappa$, one can define an $\mathbb{N}^m$-graded algebra structure of $\mathbb{R}[\Delta]$ by $\deg x_v=\bm{e}_{\kappa(v)}$, where $\bm{e}_i \in \mathbb{N}^m$ denotes the $i$th unit coordinate vector. 
Stanley~\cite{Sta} showed that, for an $\bm{a}$-balanced $(d-1)$-dimensional simplicial complex $\Delta$, $\mathbb{R}[\Delta]$ has a system of parameters $\theta_1,\ldots,\theta_d$ such that exactly $a_i$ of $\theta_j$'s are of degree $\bm{e}_i$ for $i \in [m]$. Such a system of parameters is called an \emph{$\bm{a}$-colored s.o.p.}~for $\mathbb{R}[\Delta]$.
Cook et al.~\cite{3-polytope} showed that if $\Delta$ is a balanced simplicial $2$-sphere, there is a $(1,1,1)$-colored s.o.p.~$\Theta=(\theta_1,\theta_2,\theta_3)$ for $\mathbb{R}[\Delta]$ and a linear form $\omega \in \mathbb{R}[\Delta]_1$ such that the multiplication map $(\times \omega):(\mathbb{R}[\Delta]/\Theta\mathbb{R}[\Delta])_1 \rightarrow (\mathbb{R}[\Delta]/\Theta\mathbb{R}[\Delta])_2$ is bijective.
We conjecture that the statement of Cook et al.~\cite{3-polytope} holds for a larger class of simplicial complexes and any $\bm{a} \in \mathbb{Z}_{>0}^m$ with a few exceptions as follows. 
\begin{conj} \label{conj:wl}
For $\bm{a} \in \mathbb{Z}_{>0}^m$ with $d=\sum_{i=1}^m a_i \geq 3$ and $\bm{a}\neq (d-1,1),(1,d-1)$, let $\Delta$ be an $\bm{a}$-balanced minimal $(d-1)$-cycle complex.
Then there is an $\bm{a}$-colored s.o.p.~$\Theta=(\theta_1,\ldots,\theta_d)$ for $\mathbb{R}[\Delta]$ and a linear form $\omega \in \mathbb{R}[\Delta]_1$ such that the multiplication map $(\times \omega):(\mathbb{R}[\Delta]/\Theta\mathbb{R}[\Delta])_1 \rightarrow (\mathbb{R}[\Delta]/\Theta\mathbb{R}[\Delta])_2$ is injective.
\end{conj}

A correspondence between Stanley-Reisner ring theory and (skeletal) rigidity theory has been noted in~\cite{Lee,TWW}.
In this correspondence, under the normalization $\omega=\sum_{v \in V(\Delta)} x_v$, a linear system of parameters $(\theta_1,\ldots,\theta_d)$ for $\mathbb{R}[\Delta]$ is identified with a point configuration $p:V(\Delta) \rightarrow \mathbb{R}^d$ through $\theta_i=\sum_{v\in V(\Delta)} p(v)_i x_v$ for $i \in [d]$ (see Corollary~\ref{cor:inj}). 
Thus Conjecture~\ref{conj:wl} is equivalently formulated as the problem of finding an infinitesimally rigid realization of $\bm{a}$-balanced minimal cycle complexes with non-generic point configurations.
Let $G$ be a graph, $\kappa:V(G) \rightarrow [m]$ a map, and $\bm{a} \in \mathbb{Z}_{>0}^m$ an integer vector with $\sum_{i=1}^m a_i=d$. 
We consider $\mathbb{R}^d$ as the direct product of $\mathbb{R}^{a_i}$ over all $i =1,\ldots,m$. Then each $x \in \mathbb{R}^d$ is denoted by $x=(x_1,\ldots,x_m)$ with $x_i \in \mathbb{R}^{a_i}$.
For $i =1,\ldots,m$, let $H_i:=\{x=(x_1,\ldots,x_m): x_j=0~ (j \neq i) \} \subseteq \mathbb{R}^d$. We say that a point configuration $p:V(G) \rightarrow \mathbb{R}^d$ is \emph{$(\kappa,\bm{a})$-sparse} if $p(v) \in H_{\kappa(v)}$ for all $v\in V(G)$, and $G$ is \emph{$(\kappa,\bm{a})$-sparse rigid} if $(G,p)$ is infinitesimally rigid in $\mathbb{R}^d$ for some $(\kappa,\bm{a})$-sparse point configuration $p$.
For example, if $\bm{a}=(1,\ldots,1)$ and $\kappa$ is a $d$-coloring, then a $(\kappa,\bm{a})$-sparse point configuration is the one such that each vertex of color $i$ is realized on the $i$-th coordinate axis.
Conjecture~\ref{conj:wl} can be restated as follows.
\begin{conj} \label{conj:main}
For $\bm{a} \in \mathbb{Z}_{>0}^m$ with $d=\sum_{i=1}^m a_i \geq 3$ and $\bm{a}\neq (d-1,1),(1,d-1)$, let $\Delta$ be an $\bm{a}$-balanced minimal $(d-1)$-cycle complex and $\kappa$ be an $\bm{a}$-coloring of $\Delta$.
Then $G(\Delta)$ is $(\kappa,\bm{a})$-sparse rigid.
\end{conj}
We will explain the equivalence of Conjecture~\ref{conj:wl} and Conjecture~\ref{conj:main} in Section~\ref{sec:5}.
The assumption on $\bm{a}$ in Conjecture~\ref{conj:main} is necessary.
Cook et al.~\cite{3-polytope} pointed out that there is a $(2,1)$-balanced simplicial $2$-sphere $\Delta$ and its $(2,1)$-coloring $\kappa$ such that $G(\Delta)$ is not $(\kappa,\bm{a})$-sparse rigid. As we will see in Example~\ref{eg:(d-1,1)}, the construction of Cook et al.~\cite{3-polytope} can be extended to higher dimension.
In this paper we confirm Conjecture~\ref{conj:main} for the following cases.
In Theorem~\ref{thm:a-bal}, we prove that Conjecture~\ref{conj:main} holds if $a_i \geq 2$ for all $i \in [m]$ using Fogelsanger's idea.
In Theorem~\ref{thm:C_d}, we then apply the standard coning argument and show that Conjecture~\ref{conj:main} holds if $d \geq 4$ and $\Delta$ is a homology $(d-1)$-manifold. The base cases in the coning argument are Gluck's theorem~\cite{Glu}, the result of Cook et al.~\cite[Theorem 1.1]{3-polytope} and Theorem~\ref{thm:a-bal} for $\bm{a}=(2,2)$.

The paper is organized as follows. In Section 2, preliminaries on simplicial complexes and rigidity are given. 
In Section 3, we summarize the basic properties of Fogelsanger decomposition for minimal cycle complexes.
In Section 4, we prove the rigidity of subgraph of $G(\Delta)$ induced by colors and derive the balanced lower bound theorem for a balanced minimal cycle complex $\Delta$. 
In Section 5, we discuss the equivalence of Conjecture~\ref{conj:wl} and Conjecture~\ref{conj:main}.
In Section 6, we prove Conjecture~\ref{conj:main} in the case when $a_i \geq 2$ for all $i$. 
In Section 7, we prove Conjecture~\ref{conj:main} for homology manifolds.
In Section 8, we give further observations related to Conjecture~\ref{conj:main}.

\section{Preliminaries}
\subsection{Graphs and simplicial complexes}
Throughout the paper, we only consider simple graphs and use the following basic notations.
For a graph $G$, the vertex set and the edge set is denoted by $V(G)$ and $E(G)$. 
For $X \subseteq V(G)$, $G[X]=(X,E[X])$ denotes the induced subgraph of $G$ by $X$. 
For $v \in V(G)$, $N_G(v)$ denotes the set of vertices adjacent to $v$ in $G$. 
Given a graph $G$ and an edge $uv \in E(G)$, we denote $G/uv$ to denote the simple graph obtained from $G$ by contracting $v$ onto $u$. More precisely $V(G/uv)= V-v$, $E(G/uv)= E[V-v] \cup \{uw: w\in N_G(v)\}$.

A \emph{simplicial complex} $\Delta$ is a finite collection of finite sets such that if $F \in \Delta$ and $G \subseteq F$, $G \in \Delta$.
Elements of $\Delta$ are called \emph{faces} of $\Delta$. The \emph{dimension} of a face $F \in \Delta$ is $\dim F:=|F|-1$, and a face of dimension $i$ is called an \emph{$i$-face}. 
The dimension of $\Delta$ is $\dim \Delta:=\max\{\dim F: F \in \Delta\}$.
A \emph{facet} of $\Delta$ is a maximal face under inclusion and $\Delta$ is \emph{pure} if all facets have the same dimension.
For a finite collection $\mathcal{S}$ of finite sets, the simplicial complex \emph{spanned by $\mathcal{S}$} is $\langle \mathcal{S} \rangle:=\{G \subseteq F:F \in \mathcal{S}\}$.
The vertex set $V(\Delta)$ (resp.~the edge set $E(\Delta)$) of $\Delta$ is the set of all $0$-faces (resp.~$1$-faces).
$G(\Delta)=(V(\Delta),E(\Delta))$ is called the \emph{graph} (or \emph{$1$-skeleton}) of $\Delta$.
The \emph{link} of a face $F \in \Delta$ is $\lk_\Delta(F)=\{G \in \Delta:F\cap G=\emptyset,F \cup G \in \Delta\}$. The (closed) \emph{star} of a face $F \in \Delta$ is $\st_\Delta(F)=\{G \in \Delta:F \cup G \in \Delta\}$. For $W \subseteq V(\Delta)$, define the \emph{restriction} of $\Delta$ to $W$ to be $\Delta[W]:=\{F \in \Delta:F \subseteq W\}$.

Let $\Delta_1, \Delta_2$ be pure simplicial complexes with $V(\Delta_1)  \cap V(\Delta_2) = \emptyset$ and $\dim \Delta_1 = \dim \Delta_2$.
Let $F_1$ and $F_2$ be a facet of $\Delta_1$ and $\Delta_2$ respectively, and let $\gamma:F_1 \rightarrow F_2$ be a bijection.
The \emph{connected sum} of $\Delta_1$ and $\Delta_2$, denoted as $\Delta_1 \#_\gamma \Delta_2$ (or simply $\Delta_1 \# \Delta_2$), is the simplicial complex obtained by identifying the vertices $F$ and $G$ and removing the facet corresponding to $F$ (which has been identified with $G$).
A $(d-1)$-sphere on $n (\geq d+1)$ vertices written as a $(n-d)$-fold connected sum of the boundary complex of a $d$-simplex is called a \emph{stacked $(d-1)$-sphere}.

We say that a simplicial complex $\Delta$ is a \emph{simplicial $(d-1)$-sphere} if its geometric realization is homeomorphic to $\mathbb{S}^{d-1}$.
A simplicial complex $\Delta$ is a \emph{homology $(d-1)$-manifold} over $\bm{k}$ if $\tilde{H}_{*}(\lk_\Delta(F);\bm{k})\cong\tilde{H}_{*}(\mathbb{S}^{d-|F|-1};\bm{k})$ for every nonempty face $F \in \Delta$.
Here, $\tilde{H}_*(\Delta;\bm{k})$ denotes the reduced simplicial homology group of $\Delta$ with coefficients in $\bm{k}$.
A pure $(d-1)$-dimensional simplicial complex $\Delta$ is \emph{strongly connected} if for every pair of facets $F$ and $G$ of $\Delta$, there is a sequence of facets $F=F_0,F_1\ldots,F_m=G$ such that $|F_{i-1} \cap F_{i}|=d-1$ for $i \in [m]$.
A \emph{$(d-1)$-pseudomanifold} is a strongly connected pure $(d-1)$-dimensional simplicial complex such that every $(d-2)$-face is contained in exactly two facets.
A $(d-1)$-pseudomanifold is \emph{normal} if the link of each face of dimension at most $d-3$ is connected.
The class of normal pseudomanifolds is closed under taking links~\cite{BD}.

Let $\Delta$ be a $(d-1)$-dimensional simplicial complex.
The \emph{$f$-vector} of $\Delta$ is an integer vector $f(\Delta):=(f_{-1}(\Delta),\ldots,f_{d-1}(\Delta))$, where $f_i(\Delta)$ is the number of $i$-faces of $\Delta$. 
The \emph{$h$-vector} of $\Delta$ is an integer vector $h(\Delta):=(h_0(\Delta),\ldots,h_d(\Delta))$ whose entries are given by
\[
h_j(\Delta):= \sum_{i=0}^j (-1)^{j-i} \binom{d-i}{d-j} f_{i-1}(\Delta).
\]
\subsection{Rigidity}
A \emph{framework} in $\mathbb{R}^d$ is a pair $(G,p)$ of a graph $G$ and a map $p:V(G) \rightarrow \mathbb{R}^d$. 
The map $p$ is called a \emph{point configuration}.
An \emph{infinitesimal motion} of a framework $(G,p)$ in $\mathbb{R}^d$ is a map $\dot{p}:V(G) \rightarrow \mathbb{R}^d$ satisfying 
\begin{equation} \label{eq:inf}
(p(i)-p(j))\cdot (\dot{p}(i)-\dot{p}(j))=0 \qquad ij \in E(G).
\end{equation}
An infinitesimal motion $\dot{p}$ defined by $\dot{p}(i)=Sp(i)+t$ $(i \in V(G))$ for a skew symmetric matrix $S$ and $t \in \mathbb{R}^d$ is said to be \emph{trivial}.
A framework $(G,p)$ is \emph{infinitesimally rigid in $\mathbb{R}^d$} if every infinitesimal motion of $(G,p)$ is trivial.
A graph $G$ is \emph{rigid in $\mathbb{R}^d$} if $(G,p)$ is infinitesimally rigid in $\mathbb{R}^d$ for some $p:V(G) \rightarrow \mathbb{R}^d$.

The matrix $R(G,p) \in \mathbb{R}^{|E(G)| \times d|V(G)|}$ representing (\ref{eq:inf}) is called the \emph{rigidity matrix} of $(G,p)$. More concretely, in $R(G,p)$, $d$ consecutive columns are associated to each vertex and a row associated to $ij \in E(G)$ is 
\[
    \kbordermatrix{
        &&i&&j& \\
        &\cdots 0 \cdots&p(j)-p(i)&\cdots0\cdots&p(i)-p(j)& \cdots 0 \cdots
    },
\]
where $p(i), p(j)$ are considered as a row vector.
An element in the left kernel of $R(G,p)$ is called an \emph{equilibrium stress} (or an affine $2$-stress) of $(G,p)$. We call the left kernel of $R(G,p)$ the \emph{stress space} of $(G,p)$.
If $p(V(G))$ spans at least $(d-1)$-dimensional affine subspace, a framework $(G,p)$ is infinitesimally rigid if and only if $\rank R(G,p)=d|V(G)|-\binom{d+1}{2}$.

To produce infinitesimally rigid frameworks from those with smaller size, gluing lemma and vertex splitting lemma are useful.
\begin{lemma}[Gluing Lemma] \label{lem:gluing}
Let $(G,p)$ be a framework in $\mathbb{R}^d$.
Suppose that $G_1, G_2$ are subgraphs of $G$ such that $(G_i,p|_{V(G_i)})$ is infinitesimally rigid in $\mathbb{R}^d$ for $i=1,2$ and $p(V(G_1) \cap V(G_2))$ spans at least $(d-1)$-dimensional affine subspace.
Then $(G,p)$ is infinitesimally rigid in $\mathbb{R}^d$.
\end{lemma}

\begin{lemma}[Vertex Splitting Lemma~\cite{Whi}] \label{lem:vertex splitting}
Let $G$ be a graph and $uv \in E(G)$ be an edge satisfying $|N_G(u) \cap N_G(v)| \geq d-1$.
Let $C$ be a size $(d-1)$ subset of $N_G(u) \cap N_G(v)$.
Suppose that there is $p:V(G/uv) \rightarrow \mathbb{R}^d$ such that $(G/uv,p)$ is infinitesimally rigid in $\mathbb{R}^d$ and $\{p(w)-p(u):w \in C\}$ is linearly independent.
Let $z \in \mathbb{R}^d$ be a vector not in the linear span of $\{p(w)-p(u):w \in C\}$.

Then there is $t \in \mathbb{R}$ such that for the extension $p':V(G) \rightarrow \mathbb{R}^d$ of $p$ defined by $p'(v)=p(u)+tz$, $(G,p')$ is infinitesimally rigid in $\mathbb{R}^d$.
\end{lemma}

Coning is a way to produce an infinitesimally rigid framework from that in one less dimension.
The \emph{cone graph} $G * \{v\}$ of a graph $G$ is obtained from $G$ by adding a new vertex $v$ and edges from $v$ to all the original vertices in $G$.

\begin{lemma}[Cone Lemma~\cite{Whi2}]\label{lem:coning}
Let $G$ be a graph and $G * \{v\}$ be its cone graph.
Let $p:V(G) \cup \{v\} \rightarrow \mathbb{R}^{d+1}$ be a point configuration of cone graph such that $p(v) \neq p(u)$ for all $u \in V(G)$.
Let $H \cong \mathbb{R}^d$ be a hyperplane in $\mathbb{R}^{d+1}$ not passing through $p(v)$ and not parallel to the vectors $\{p(u)-p(v):u \in V(G)\}$.
Define a projection $p_H:V(G) \rightarrow \mathbb{R}^d$ of $p$ by letting $p_H(u)$ be the intersection point of $H$ and the line passing through $p(v)$ and $p(u)$.
Then $(G * \{v\},p)$ is infinitesimally rigid in $\mathbb{R}^{d+1}$ if and only if $(G,p_H)$ is infinitesimally rigid in $\mathbb{R}^d$.
\end{lemma}

\section{Decomposition of minimal cycle complexes} \label{sec:3}
Let $\Delta$ be a simplicial complex and $uv \in E(\Delta)$.
We define a simplicial complex $\Delta/uv$ by
\[
\Delta/uv=\{F \in \Delta: v \not \in F\} \cup \{F-v+u: F \in \st_\Delta(v)\}.
\]
The operation $\Delta \rightarrow \Delta/uv$ is called an \emph{edge contraction} of $uv$ (onto $u$).
We remark that we do not allow multiple copies of the same faces, so for $G \in \lk_\Delta(u)\cap\lk_\Delta(v)$, faces $G+u$ and $G+v$ of $\Delta$ are identified under the edge contraction of $uv$.

Let $\Gamma$ be an abelian group and let $\Delta$ be a simplicial complex on the vertex set $[n]$.
An \emph{$i$-chain} is a $\Gamma$-coefficient formal sum of $i$-faces.
For an $i$-chain $c=\sum_{F}a_F F$, define an $(i-1)$-chain $\partial c$ by $\partial c := \sum_{F} a_F \partial F$, where for $F=\{x_1,\ldots,x_{i+1}\}$ ($x_1< \cdots <x_{i+1}$), $\partial F:=\sum_{j=1}^{i+1} (-1)^j F \setminus \{x_j\}$.
An \emph{$i$-cycle} is an $i$-chain $c$ satisfying $\partial c=0$.
An $i$-chain $c'=\sum_{F} b_F F$ is a \emph{subchain} of an $i$-chain $c=\sum_{F} a_F F$ if for each $i$-face $F$, $b_F=a_F$ or $b_F=0$ holds.
An $i$-cycle $c$ is a \emph{minimal $i$-cycle} if its only subchains which are $i$-cycles are $c$ and $0$. The \emph{support} of an $i$-chain $c=\sum_{F}a_F F$ is $\supp c := \{F:a_F \neq 0\} \subseteq \binom{[n]}{i+1}$, where $\binom{[n]}{i+1}$ is the set of all $(i+1)$-subsets of $[n]$.

A $(d-1)$-dimensional simplicial complex $\Delta$ is called a \emph{minimal $(d-1)$-cycle complex} (over $\Gamma$) if (i) $\Delta=\langle \{F\} \rangle$ for a $d$-set $F$ or (ii) $\Delta=\langle \supp c \rangle$ for a nonzero minimal $(d-1)$-cycle $c$.
If the dimension is clear from the context, it is simply called a minimal cycle complex.
We remark that although minimal cycle complexes are originally defined as simplicial complexes satisfying (ii), we include the case (i) to make the presentation simple. It is possible to unify (i) and (ii) by using multicomplex formulation as in \cite{CJT}.
A minimal cycle complex satisfying (i) (resp.~(ii)) is said to be \emph{trivial} (resp.~\emph{nontrivial}).
A minimal cycle complex is always pure.
Minimal cycle complexes may have singular point, so the class of minimal cycle complexes are not closed under taking links.

Pseudomanifolds are minimal cycle complexes over $\mathbb{Z}_2$~\cite{CJT}. 
$2$-CM complexes over a field $\bm{k}$ are a minimal cycle complexes over a free $\bm{k}$-module of finite rank~\cite{Nev}.
Other examples of minimal cycle complexes arise in the context of simplicial matroid.
A \emph{$(d-1)$-simplicial matroid} on a subset $\mathcal{E}$ of $\binom{[n]}{d}$ over a field $\bm{k}$ is a matroid such that $\{F_1,\ldots,F_k\} \subseteq \mathcal{E}$ is independent if and only if $\{\partial F_1,\ldots, \partial F_k\}$ is linearly independent.
A simplicial complex spanned by the circuit of $(d-1)$-simplicial matroid over $\bm{k}$ is a minimal $(d-1)$-cycle complex over $\bm{k}$.
A matroid on a ground set $E$ is \emph{connected} if for any $e,f \in E$, there is a circuit containing both $e$ and $f$.
A pure $(d-1)$-dimensional simplicial complex whose facets define a connected $(d-1)$-simplicial matroid over $\bm{k}$ is a minimal $(d-1)$-cycle complex over a free $\bm{k}$-module of finite rank. 
Thus a simplicial complex having a convex ear decomposition~\cite{Cha} is a minimal cycle complex.

We list basic properties of minimal cycle complexes.
\begin{lemma} \label{lem:basic}
For a nontrivial minimal $(d-1)$-cycle complex $\Delta$, the followings hold:
\begin{itemize}
\item[(i)] $\Delta$ is strongly connected.
\item[(ii)] Every $(d-2)$-face of $\Delta$ is contained in at least two facets of $\Delta$. In particular, $|V(\Delta)| \geq d+1$.
\end{itemize}
\end{lemma}
\begin{proof}
Let $c$ be a minimal $(d-1)$-cycle satisfying $\Delta=\langle \supp c \rangle$. 

(i) Suppose that $\Delta$ is not strongly connected. Then there is a proper subset $\mathcal{S}$ of $\supp c$ such that no pair of $F \in \mathcal{S}$ and $G \in \supp c \setminus \mathcal{S}$ shares the same $(d-2)$-face.
Then, $c$ restricted to $\mathcal{S}$ is a $(d-1)$-cycle and is a proper subchain of $c$, which is a contradiction.

(ii) Suppose that a $(d-2)$-face $T$ of $\Delta$ is contained in only one facet $F$ of $\Delta$. Then $(\partial c)_T=\pm c_F \neq 0$, which contradicts the fact that $c$ is a $(d-1)$-cycle. 
\end{proof}

Fogelsanger~\cite{Fog} pointed out that a minimal $(d-1)$-cycle can be decomposed into minimal $(d-1)$-cycles in such a way that it behaves nicely with respect to edge contractions. 
We summarize the properties of this decomposition below, which is necessary in the later discussion, in terms of minimal cycle complexes. See \cite{CJT,Fog} for more details.
\begin{lemma}[Fogelsanger~\cite{Fog}] \label{lem:Fog_dec}
Suppose that $\Delta$ is a nontrivial minimal $(d-1)$-cycle complex and $uv \in E(\Delta)$.
Then there exists a sequence of nontrivial minimal $(d-1)$-cycle complexes $\Delta_1^+,\ldots,\Delta_m^+$ satisfying the following properties:
\begin{itemize}
\item[(a)] For each $i$, $uv \in E(\Delta_i^+)$. Equivalently, there is a facet in $\Delta_i^+$ containing both $u$ and $v$.
\item[(b)] For each $i$ and each facet $F$ of $\Delta_i^+$ with $F \not\in\Delta$, $u$ and $v$ are contained in $F$, and $F-u$ and $F-v$ are $(d-2)$-faces of $\Delta_i^+ \cap \Delta$.
\item[(c)] For each $i$, $\Delta_i^+/uv$ is a minimal $(d-1)$-cycle complex.
\item[(d)] $G(\Delta)=\bigcup_{i=1}^m G(\Delta_i^+)$.
\item[(e)] For each $i \geq 2$, $(\bigcup_{j <i} \Delta_j^+)$ and $\Delta_i^+$ share at least one facet.
\end{itemize}
\end{lemma}
We remark that, as in Lemma~\ref{lem:Fog_dec} (b), for the minimal cycle complexes $\Delta_1^+,\ldots,\Delta_m^+$ given in Lemma~\ref{lem:Fog_dec}, $\Delta_i^+$ may have a face not in $\Delta$.
We also remark that the decomposition given in Lemma~\ref{lem:Fog_dec} is not unique.

\section{Rigidity of rank-selected subcomplexes}
We use the notation $[d]:=\{1,\ldots,d\}$.
A $(d-1)$-dimensional simplicial complex $\Delta$ is \emph{balanced} (or \emph{completely balanced}) if $G(\Delta)$ is $d$-colorable, or equivalently, there is a map $\kappa:V(\Delta) \rightarrow [d]$ such that for any face $F \in \Delta$, $|F \cap \kappa^{-1}(i)| \leq 1$ for $i \in [d]$.
Such a coloring is called a \emph{proper coloring} of $\Delta$.
A proper coloring of a strongly connected simplicial complex is unique up to the permutation of colors and is called \emph{the proper coloring} of $\Delta$.
The boundary complex of a $d$-dimensional \emph{cross-polytope} $\mathcal{C}_d^*:=\conv\{\pm\bm{e}_i:i \in [d]\}$, where $\bm{e}_i$ is the $i$th unit coordinate vector in $\mathbb{R}^d$, is an example of a balanced $(d-1)$-simplicial complex.

Let $\Delta$ be a balanced, strongly connected, $(d-1)$-dimensional simplicial complex and $\kappa$ be the proper coloring of $\Delta$. For $T \subseteq [d]$, the \emph{$T$-rank selected subcomplex} of $\Delta$ is $\Delta_T:=\Delta[\kappa^{-1}(T)]$.
Klee and Novik~\cite{KN} showed that $G(\Delta_T)$ is rigid in $\mathbb{R}^{|T|}$ for any $T\subseteq[d]$ with $|T| \geq 3$ if $\Delta$ is a balanced normal $(d-1)$-pseudomanifold with $d \geq 3$. 
We extend this result to the class of minimal $(d-1)$-cycle complexes as follows.
\begin{theorem} \label{thm:rank-selcted}
For $d\geq 3$, let $\Delta$ be a balanced minimal $(d-1)$-cycle complex.
Then for $T \subseteq [d]$ with $|T| \geq 3$, $G(\Delta_T)$ is rigid in $\mathbb{R}^{|T|}$.
\end{theorem}

For a pure simplicial complex $\Delta$, we say that a vertex subset $U \subseteq V(\Delta)$ is \emph{$(\geq k)$-transversal} if $|U \cap F| \geq k$ for every facet $F \in \Delta$.
Since, for a proper coloring $\kappa$ of a pure simplicial complex $\Delta$ and $T \subseteq [d]$, $\kappa^{-1}(T) \subseteq V(\Delta)$ is $(\geq |T|)$-transversal, Theorem~\ref{thm:rank-selcted} is a corollary of the following lemma.
\begin{lemma}
For $d \geq k \geq 3$, let $\Delta$ be a minimal $(d-1)$-cycle complex and $U \subseteq V(\Delta)$ be a $(\geq k)$-transversal set of $\Delta$.
Then $G(\Delta[U])$ is rigid in $\mathbb{R}^{k}$.
\end{lemma}
\begin{proof}
We prove the statement by the induction on $|V(\Delta)|$.
If $\Delta$ is a trivial $(d-1)$-cycle complex, $G(\Delta)$ is a complete graph. Hence, $G(\Delta[U])=G(\Delta)[U]$ is a complete graph, which is rigid in $\mathbb{R}^{k}$.

Suppose that $\Delta$ is a nontrivial minimal $(d-1)$-cycle complex. Pick $u,v \in U$ with $uv \in E(\Delta)$. Such $u,v$ always exist by $k \geq 3$.
Let $\Delta_1^+,\ldots,\Delta_m^+$ be the nontrivial minimal $(d-1)$-cycle complexes given in Lemma~\ref{lem:Fog_dec} with respect to $\Delta$ and $uv$. 
For each $i$, let $U_i:=U \cap V(\Delta_i^+)$.
For each facet $F$ of $\Delta_i^+ \cap \Delta$, we have $|F \cap U_i| \geq k$.
For each facet $F \in \Delta_i^+ \setminus \Delta$, as $F$ contains $u$ and $F-u$ is a $(d-2)$-face of $\Delta$ by Lemma~\ref{lem:Fog_dec} (b), we have $|F \cap U_i|=|(F-u) \cap U_i|+1 \geq k$. 
Hence $U_i$ is a $(\geq k)$-transversal set of $\Delta_i^+$.

\begin{claim} \label{claim:chap4}
$|N_{G(\Delta_i^+)}(u) \cap N_{G(\Delta_i^+)}(v) \cap U_i| \geq k-1$ holds for each $i$.
\end{claim}
\begin{proof}[Proof of claim]
Let $F$ be a facet of $\Delta_i^+$ containing both $u$ and $v$. Such a facet always exists by Lemma~\ref{lem:Fog_dec}(a). 
As $U_i$ is a $(\geq k)$-transversal set of $\Delta_i^+$, we have $|F \cap U_i| \geq k$.
Thus we have $|(F-u-v) \cap U_i| \geq k-2$. 
If $|(F-u-v) \cap U_i| \geq k-1$, the claim follows as $(F-u-v)\cap U_i$ is included in $N_{G(\Delta_i^+)}(u) \cap N_{G(\Delta_i^+)}(v)$. 
If $|(F-u-v) \cap U_i|=k-2$, by $k \geq 3$, we can pick $w \in (F-u-v) \cap U_i$. 
By Lemma~\ref{lem:basic}(ii), there is another facet $G \in \Delta_i^+$ different from $F$ which includes $F-w$. Since $|G \cap U_i| \geq k$, the unique element $x \in G \setminus F$ must be in $U_i$. Now $(F -u - v +x)\cap U_i$ is the desired subset.
\end{proof}
For each $i$, since $u,v \in U_i$, $U_i-v$ is a $(\geq k)$-transversal set of $\Delta_i^+/uv$ and $G(\Delta_i^+[U_i])/uv=G((\Delta_i^+/uv)[U_i-v])$.
Also $\Delta_i^+/uv$ is a minimal $(d-1)$-cycle complex by Lemma~\ref{lem:Fog_dec} (c).
Hence, by the induction hypothesis, $G((\Delta_i^+/uv)[U_i-v])$ is rigid in $\mathbb{R}^k$.
By Lemma~\ref{lem:vertex splitting} and Claim~\ref{claim:chap4}, $G(\Delta_i^+[U_i])$ is rigid in $\mathbb{R}^k$.
Now the rigidity of $G(\Delta[U])=\bigcup_{i=1}^m G(\Delta_i^+[U_i])$ follows from Lemma~\ref{lem:gluing} and Lemma~\ref{lem:Fog_dec} (d), (e).
\end{proof}

The balanced analogue of the lower bound theorem~\cite{Bar} has been investigated in~\cite{GKN,KN}.
For positive integers $n,d$ with $n$ divisible by $d$, a \emph{stacked cross-polytopal $(d-1)$-sphere} on $n$ vertices is the connected sum of $\frac{n}{d}-1$ copies of the boundary complex of the cross-polytope $\mathcal{C}_d^*$.
A stacked cross-polytopal $(d-1)$-sphere $\Delta$ satisfies $2h_2(\Delta)=(d-1)h_1(\Delta)$, and the balanced lower bound theorem~\cite{KN} asserts that $2h_2(\Delta)\geq(d-1)h_1(\Delta)$ holds for any balanced normal $(d-1)$-pseudomanifolds with $d \geq 3$.
Goff, Klee, Novik~\cite{GKN} showed that for a balanced pure simplicial complex $\Delta$, this inequality follows from the rigidity of $G(\Delta_T)$ in $\mathbb{R}^3$ for every $3$-set $T \subseteq [d]$.
Hence Theorem~\ref{thm:rank-selcted} implies the following generalization of balanced lower bound theorem to balanced minimal cycle complexes.
\begin{corollary} \label{cor:balanced_LBT}
Let $\Delta$ be a balanced minimal $(d-1)$-cycle complex with $d \geq3$. Then $2h_2(\Delta) \geq (d-1)h_1(\Delta)$.
\end{corollary}

\begin{remark}
Characterizing simplicial complexes achieving the tight equality in the lower bound theorem is also a well-studied problem.
Klee and Novik~\cite{KN} showed that for a balanced normal $(d-1)$-pseudomanifold $\Delta$ with $d \geq 4$, $2h_2(\Delta) = (d-1)h_1(\Delta)$ holds if and only if $\Delta$ is a stacked cross-polytopal $(d-1)$-sphere. 
It is interesting to know when the equality $2h_2(\Delta) = (d-1)h_1(\Delta)$ occurs for a balanced minimal $(d-1)$-cycle complex $\Delta$.
\end{remark}



\section{Stanley-Reisner ring and rigidity} \label{sec:5}
For an $\mathbb{N}$-graded algebra $A$, the $i$th homogeneous component is denoted as $A_i$. An element of $A_1$ is called a linear form.
For a sequence of linear forms $\Theta=(\theta_1,\ldots,\theta_d)$ of $\mathbb{R}[\Delta]$, the ideal of $\mathbb{R}[\Delta]$ generated by $\theta_1,\ldots,\theta_d$ is denoted as $\Theta \mathbb{R}[\Delta]$.
For a $(d-1)$-dimensional simplicial complex $\Delta$, a sequence of $d$ linear forms $\theta_1,\ldots,\theta_d \in \mathbb{R}[\Delta]_1$ is called a \emph{linear system of parameters} (\emph{l.s.o.p.}~for short) for $\mathbb{R}[\Delta]$ if $\dim_{\mathbb{R}} \mathbb{R}[\Delta]/\Theta\mathbb{R}[\Delta] < \infty$.
For Stanley-Reisner ring, the following criterion of an l.s.o.p.~is known (see~\cite{Sta}).
\begin{lemma} \label{lem:Green-Kleitman}
Let $\Delta$ be a $(d-1)$-dimensional simplicial complex and $\theta_1,\ldots,\theta_d \in \mathbb{R}[\Delta]_1$ be linear forms. Define $p:V(\Delta) \rightarrow \mathbb{R}^d$ by the relation $\theta_i=\sum_{v \in V(\Delta)}p(v)_ix_v$ for $i \in [d]$.
Then $\theta_1,\ldots,\theta_d$ is an l.s.o.p.~for $\mathbb{R}[\Delta]$ if and only if $\{p(v):v \in F\}$ is linearly independent for every $F \in \Delta$.
\end{lemma}

The connection between Stanley-Reisner ring theory and rigidity theory was pointed out by Lee~\cite{Lee}.
Lee~\cite{Lee} defined the notion of linear and affine $r$-stresses of a simplicial complex $\Delta$ and proved that, for a sequence of linear forms $\Theta=(\theta_1,\ldots,\theta_d)$ and a linear form $\omega=\sum_{v \in V(\Delta)} x_v$, $\left(\mathbb{R}[\Delta]/\Theta \mathbb{R}[\Delta]\right)_r$ (resp.~$\left(\mathbb{R}[\Delta]/(\Theta,\omega) \mathbb{R}[\Delta]\right)_r$) is linearly isomorphic to the space of linear (resp.~affine) $r$-stresses of $\Delta$.
For the remaining argument, the case of $r=1,2$ is related, which is summarized as below. 
\begin{lemma}[Lee~\cite{Lee}] \label{lem:Lee}
Let $\Delta$ be a $(d-1)$-dimensional simplicial complex.
Let $\Theta=(\theta_1,\ldots,\theta_d)$ be a sequence of linear forms of $\mathbb{R}[\Delta]$.
Let $\omega=\sum_{v \in V(\Delta)}x_v$ and define $p:V(\Delta) \rightarrow \mathbb{R}^d$ by the relation $\theta_i=\sum_{v \in V(\Delta)}p(v)_ix_v$ for $i \in [d]$. 
The followings hold:
\begin{itemize}
\item[(i)] $(\mathbb{R}[\Delta]/\Theta \mathbb{R}[\Delta])_1$ is linearly isomorphic to $\{t \in \mathbb{R}^{V(\Delta)}:\sum_{v \in V(\Delta)}t_v p(v)=0\}$ .
\item[(ii)] $\left(\mathbb{R}[\Delta]/(\Theta,\omega) \mathbb{R}[\Delta]\right)_2$ is linearly isomorphic to $\ker R(G(\Delta),p)^\top$.
\item[(iii)] $(\mathbb{R}[\Delta]/\Theta \mathbb{R}[\Delta])_2$ is linearly isomorphic to $\ker R(G(\Delta)*\{u\},p')^\top$, where $G(\Delta)*\{u\}$ is the cone graph of $G(\Delta)$ and $p':V(\Delta)\cup\{u\} \rightarrow \mathbb{R}^d$ is the extension of $p$ defined by $p'(u)=0$.
\end{itemize}
\end{lemma}
We have the following corollary. Although the result is already known, we include the proof for completeness.
\begin{corollary}\label{cor:inj}
Let $\Delta$ be a strongly connected $(d-1)$-dimensional simplicial complex.
Let $\Theta=(\theta_1,\ldots,\theta_d)$ be an l.s.o.p.~for $\mathbb{R}[\Delta]$.
Let $\omega=\sum_{v \in V(\Delta)}x_v$ and define $p:V(\Delta) \rightarrow \mathbb{R}^d$ by the relation $\theta_i=\sum_{v \in V(\Delta)}p(v)_ix_v$ for $i \in [d]$.
Then the multiplication map $(\times \omega):\left(\mathbb{R}[\Delta]/\Theta \mathbb{R}[\Delta]\right)_1 \rightarrow \left(\mathbb{R}[\Delta]/\Theta \mathbb{R}[\Delta]\right)_2$ is injective if and only if $(G(\Delta),p)$ is infinitesimally rigid in $\mathbb{R}^d$.
\end{corollary}
\begin{proof}
Since $\Theta$ is an l.s.o.p.~for $\mathbb{R}[\Delta]$, $p(V(\Delta))$ linearly spans $\mathbb{R}^d$ by Lemma~\ref{lem:Green-Kleitman}.
Hence by Lemma~\ref{lem:Lee} (i), 
$\dim \left(\mathbb{R}[\Delta]/\Theta \mathbb{R}[\Delta]\right)_1 = f_0(\Delta)-d=h_1(\Delta)$.
By Lemma~\ref{lem:Lee} (ii), $\Coker(\times \omega) \cong \left(\mathbb{R}[\Delta]/(\Theta,\omega) \mathbb{R}[\Delta]\right)_2$ is linearly isomorphic to $\ker R(G(\Delta),p)^\top$.
Hence $\dim \Coker(\times \omega)=f_1(\Delta)-df_0(\Delta)+\binom{d+1}{2}(=h_2(\Delta)-h_1(\Delta))$ if and only if $(G(\Delta),p)$ is infinitesimally rigid in $\mathbb{R}^d$.
Thus, the statement follows if $\dim (\mathbb{R}[\Delta]/\Theta \mathbb{R}[\Delta])_2=h_2(\Delta)$ always holds. 

To see this, let $G(\Delta) * \{u\}$ be the cone graph of $G(\Delta)$ and let $p':V(\Delta) \cup \{u\} \rightarrow \mathbb{R}^d$ be the point configuration as in Lemma~\ref{lem:Lee} (iii).
By Lemma~\ref{lem:Green-Kleitman}, for each facet $F$ of $\Delta$, $F + u$ is a clique in $G(\Delta) * \{u\}$ and $p'(F + u)$ affinely spans $\mathbb{R}^d$, so $(G(\Delta)[F+u],p'|_{F+u})$ is infinitesimally rigid in $\mathbb{R}^d$.
Again by Lemma~\ref{lem:Green-Kleitman}, for facets $F$ and $G$ of $\Delta$ with $|F \cap G|=d-1$, $p'((F\cap G)+u)$ affinely spans a $(d-1)$-dimensional subspace. 
As $\Delta$ is strongly connected, one can order the facets of $\Delta$ as $F_1,\ldots,F_m$ in such a way that, for each $i \geq 2$, there is $j <i$ satisfying $|F_i \cap F_j|=d-1$.
Hence by the repeated application of Lemma~\ref{lem:gluing}, $(G(\Delta)*\{u\},p')$ is infinitesimally rigid in $\mathbb{R}^d$.
Therefore by Lemma~\ref{lem:Lee}(iii), we get $\dim (\mathbb{R}[\Delta]/\Theta \mathbb{R}[\Delta])_2 =\dim \ker R(G(\Delta)*\{u\},p')^\top = (f_1(\Delta)+f_0(\Delta))-d(f_0(\Delta)+1) + \binom{d+1}{2}=h_2(\Delta)$ as desired.
\end{proof}

Let us recall the definition of $\bm{a}$-balancedness.
For $\bm{a} \in \mathbb{Z}_{>0}^m$ with $\sum_{i=1}^m a_i=d$, a $(d-1)$-dimensional simplicial complex $\Delta$ on the vertex set $V(\Delta)$ is \emph{$\bm{a}$-balanced} if there is a map $\kappa:V(\Delta) \rightarrow [m]$ satisfying $|F \cap \kappa^{-1}(i)|\leq a_i$ for any $F \in \Delta$.
We call such a map $\kappa$ an \emph{$\bm{a}$-coloring} of $\Delta$.
If $\Delta$ is pure, this condition is equivalent to $|F \cap \kappa^{-1}(i)| = a_i$ for any facet $F$ of $\Delta$.
Stanley~\cite{Sta} showed that an $\bm{a}$-balanced simplicial complex $\Delta$ admits a special type of l.s.o.p.~for $\mathbb{R}[\Delta]$. 
We remark that we always use the term ``linear forms'' to mean degree one forms in the usual $\mathbb{N}$-grading.
\begin{prop}[Stanley~\cite{Sta}] \label{prop:a-bal-lsop}
Let $\Delta$ be an $\bm{a}$-balanced $(d-1)$-dimensional simplicial complex for $\bm{a} \in \mathbb{Z}_{>0}^m$ and $\kappa:V(\Delta)\rightarrow [m]$ be an $\bm{a}$-coloring of $\Delta$.
Make $\mathbb{R}[\Delta]$ into an $\mathbb{N}^m$-graded algebra by defining $\deg x_v = \bm{e}_{\kappa(v)} \in \mathbb{N}^m$.
Then $\mathbb{R}[\Delta]$ has an l.s.o.p.~$\Theta=(\theta_1,\ldots,\theta_d)$ such that each $\theta_i$ is homogeneous in $\mathbb{N}^m$-grading and exactly $a_i$ elements among $\Theta$ are of degree $\bm{e}_i$ for each $i \in [m]$.
\end{prop}

An l.s.o.p.~satisfying the property in Proposition~\ref{prop:a-bal-lsop} is called an \emph{$\bm{a}$-colored s.o.p.} We now prove the equivalence of Conjecture~\ref{conj:wl} and Conjecture~\ref{conj:main}.
We recall the definition of $(\kappa,\bm{a})$-sparse rigidity.
Let $G$ be a graph, $\kappa:V(G) \rightarrow [m]$ a map, and $\bm{a} \in \mathbb{Z}_{>0}^m$ an integer vector with $\sum_{i=1}^m a_i=d$. 
Let $H_i:=0 \times \cdots \times \mathbb{R}^{a_i} \times \cdots \times 0 \subseteq \mathbb{R}^d=\prod_{i=1}^m \mathbb{R}^{a_i}$ for $i \in [m]$. We say that a point configuration $p:V(G) \rightarrow \mathbb{R}^d$ is \emph{$(\kappa,\bm{a})$-sparse} if $p(v) \in H_{\kappa(v)}$ for all $v\in V(G)$, and $G$ is \emph{$(\kappa,\bm{a})$-sparse rigid} if $(G,p)$ is infinitesimally rigid in $\mathbb{R}^d$ for some $(\kappa,\bm{a})$-sparse point configuration $p$.
Note that by Lemma~\ref{lem:Green-Kleitman}, under the identification $\theta_i=\sum_{v \in V(\Delta)} p(v)_i x(v)$ for $i \in [d]$, $\theta_1,\ldots,\theta_d$ is an $\bm{a}$-colored s.o.p.~for $\mathbb{R}[\Delta]$ if and only if $p:V(\Delta) \rightarrow \mathbb{R}^d$ is $(\kappa,\bm{a})$-sparse and $\{p(v):v \in F\}$ is linearly independent for every $F \in \Delta$.
\begin{proof}[Proof of equivalence of Conjecture~\ref{conj:wl} and Conjecture~\ref{conj:main}]
If Conjecture~\ref{conj:wl} holds, for any generic choice of $\omega$, $(\times \omega):\left(\mathbb{R}[\Delta]/\Theta \mathbb{R}[\Delta]\right)_1 \rightarrow \left(\mathbb{R}[\Delta]/\Theta \mathbb{R}[\Delta]\right)_2$ is injective. 
So, in Conjecture~\ref{conj:wl}, we may add extra constraints that $\omega=\sum_{v \in V(\Delta)}a_v x_v$ with $a_v \neq 0$ for all $v \in V(\Delta)$.
Moreover, by setting $a_vx_v$ to $x_v$, we may further suppose $\omega=\sum_{v \in V(\Delta)}x_v$ in the statement of Conjecture~\ref{conj:wl}.
Now, the equivalence of Conjecture~\ref{conj:wl} and Conjecture~\ref{conj:main} follows from Corollary~\ref{cor:inj}.
\end{proof}

\section{$(\kappa,\bm{a})$-sparse rigidity of minimal cycle complexes for $\bm{a} \in \mathbb{Z}_{\geq 2}^m$}

In this section we shall verify Conjecture~\ref{conj:main} in the case when $a_i \geq 2$ for all $i$ as follows.
\begin{theorem} \label{thm:a-bal}
Let $\bm{a}=(a_1,\ldots,a_m) \in \mathbb{Z}_{>0}^m$ be a positive integer vector with $a_i \geq 2$ for all $i \in [m]$ and $d:=\sum_{i=1}^m a_i \geq 3$.
Let $\Delta$ be an $\bm{a}$-balanced minimal $(d-1)$-cycle complex with an $\bm{a}$-coloring $\kappa$.
Then $G(\Delta)$ is $(\kappa,\bm{a})$-sparse rigid.
\end{theorem}

The proof of Theorem~\ref{thm:a-bal} consists of several lemmas.
To state lemmas, we introduce $L$-sparse rigidity as a generalization of $(\kappa,\bm{a})$-sparse rigidity. 
For $x=(x_1,\ldots,x_d) \in \mathbb{R}^d$, we denote $\supp x:=\{i \in [d]:x_i \neq 0\}$.
Let $G$ be a graph and $L:V(G) \rightarrow 2^{[d]}$ be a map.
A point configuration $p:V(G) \rightarrow \mathbb{R}^d$ is \emph{$L$-sparse} if $\supp p(v) \subseteq L(v)$ for every $v \in V(G)$.
We say that an $L$-sparse point configuration $p:V(G) \rightarrow \mathbb{R}^d$ is \emph{generic} over $\mathbb{Q}$ if $\{p(v)_j: v \in V(G), j \in L(v)\}$ is algebraically independent over $\mathbb{Q}$.
A graph $G$ is \emph{$L$-sparse rigid} if  $(G,p)$ is infinitesimally rigid in $\mathbb{R}^d$ for some $L$-sparse point configuration $p$.
Note that for $\kappa:V(G) \rightarrow [m]$ and $\bm{a} \in \mathbb{Z}_{>0}^m$ with $d=\sum_{i=1}^m a_i$, $(\kappa,\bm{a})$-sparse rigidity coincides with $L_{\kappa,\bm{a}}$-sparse rigidity, where $L_{\kappa,\bm{a}}:V(G) \rightarrow \mathbb{R}^{[d]}$ is defined by partitioning $[d]$ into disjoint sets $I_1, \ldots, I_m$ with $|I_i|=a_i$ and letting $L_{\kappa,\bm{a}}(v)=I_{\kappa(v)}$.
A $(\kappa,\bm{a})$-sparse point configuration $p$ is said to be \emph{generic} if $p$ is generic as an $L_{\kappa,\bm{a}}$-sparse point configuration.
Eftekhari et al.~\cite{EJNSTW} addressed the special setting of $L$-sparse rigidity in which, given $X \subseteq V(G)$, $L(v)=[d-1]$ for $v \in X$ and $L(v)=[d]$ for $v \not\in X$, and they gave a combinatorial characterization when $d=2$.
Cook et al.~\cite{3-polytope} gave a combinatorial characterization for $(\kappa,\bm{a})$-sparse rigidity of maximal planar graphs $G$ with $\bm{a}=(2,1)$ and $\kappa:V(G) \rightarrow \{1,2\}$ in which $\kappa^{-1}(2)$ is a stable set in $G$.

A set of points $X \subseteq \mathbb{R}^d$ is \emph{affinely independent} if the dimension of the affine span of $X$ is $|X|-1$.
\begin{lemma} \label{lem:gen_pos}
Let $G$ be a graph and $U \subseteq V(G)$ be a subset of vertices with size at most $d+1$.
Let $L:V(G) \rightarrow 2^{[d]}$ be a map, and $p:V(G) \rightarrow \mathbb{R}^d$ be a generic $L$-sparse point configuration. 
Then $p(U)$ is affinely independent if and only if $|\bigcup_{v \in W} L(v)|+1 \geq |W|$ for every subset $W \subseteq U$.
\end{lemma}
\begin{proof}
Let $k=|U|$ and $U=\{v_1,\ldots,v_k\}$. Let $A \in \mathbb{R}^{(d+1)\times k}$ be the matrix defined by 
\begin{align*}
A=
\begin{bmatrix}
p(v_1) & \cdots & p(v_k) \\
1 & \cdots & 1 \\
\end{bmatrix}.
\end{align*}
Then $p(U)$ is affinely independent if and only if $\rank A=k$.
Consider the bipartite graph $G(A)$ between $U$ and $[d+1]$ such that there is an edge between $v \in U$ and $i \in [d+1]$ if and only if $A_{i,v} \neq 0$.

For $R \subseteq [d+1]$ with $|R|=k$, let $P_R$ be the determinant of submatrices of $A$ indexed by $R$.
Then 
\begin{equation} \label{eq:submat}
P_R=\sum_{\sigma} s_\sigma \prod_{j=1}^k A_{v_j,\sigma(v_j)},
\end{equation}
where the sum is taken over all perfect matchings $\sigma$ of $G(A)[U \sqcup R]$ and $s_\sigma=\pm 1$.
When $P_R$ is considered as a polynomial of $\{p(v)_i:v \in U,i \in L(v)\}$, each monomial appears at most once in the summand of the right hand side of (\ref{eq:submat}). Since $p$ is a generic $L$-sparse point configuration, $P_R \neq 0$ if and only if there is a perfect matching in $G(A)[U \sqcup R]$.

Hence $\rank A=k$ if and only if there is a size $|U|$ matching in the bipartite graph $G(A)$, which is equivalent to the condition given in the statement by Hall's theorem.
\end{proof}

Let $G$ be a graph and $L:V(G) \rightarrow 2^{[d]}$ be a map.
For $U \subseteq V(G)$, consider the following ``Hall condition'' (H):
\begin{itemize}
\item[(H)] $|\bigcup_{v \in W} L(v)|+1 \geq |W|$ holds for every subset $W \subseteq U$.
\end{itemize}

\begin{lemma} \label{lem:hall}
Let $G$ be a graph and $uv \in E(G)$ be an edge satisfying $|N_G(u) \cap N_G(v)| \geq d-1$. 
Let $L:V(G) \rightarrow 2^{[d]}$ be a map satisfying $L(u)=L(v)$, and define $L'$ as the restriction of $L$ to $V(G/uv)$.
Suppose that there is a $(d-1)$-set $C \subseteq N_G(u) \cap N_G(v)$ such that $C+u+v$ satisfies the condition (H) with respect to $G$ and $L$.
Then $G$ is $L$-sparse rigid if $G/uv$ is $L'$-sparse rigid.
\end{lemma}
\begin{proof}
Suppose that $G/uv$ is $L'$-sparse rigid.
Let $p:V(G/uv) \rightarrow \mathbb{R}^d$ be a generic $L'$-sparse point configuration. Then $(G/uv,p)$ is infinitesimally rigid in $\mathbb{R}^d$.
By Lemma~\ref{lem:gen_pos}, the condition (H) of $C+u+v$ implies that the set of points $\{p(w) : w\in C \cup \{u\}\}$ is affinely independent. Hence $\{p(w)-p(u): w \in C\}$ is linearly independent.

Among vectors $z \in \mathbb{R}^{d}$ with $\supp(z)=L(v)$, pick a generic one $z$.
Then, by Lemma~\ref{lem:gen_pos}, the condition (H) of $C+u+v$ implies that $\{p(w):w\in C \cup \{u\} \} \cup \{p(u)+z\}$ is affinely independent. Hence $z$ is not in the linear span of $\{p(w)-p(u):w\in C\}$.
Thus by Lemma~\ref{lem:vertex splitting}, there is an extension $p'$ of $p$ such that $p'(v)=p(u)+tz$ for some $t \in \mathbb{R}$ such that $(G,p')$ is infinitesimally rigid.
As $p'$ is an $L$-sparse point configuration, $G$ is $L$-rigid.
\end{proof}
As a special case, we have the following corollary for $(\kappa,\bm{a})$-sparse rigidity.
\begin{corollary} \label{cor:hall}
Let $G$ be a graph and $uv \in E(G)$ be an edge satisfying $|N_G(u) \cap N_G(v)| \geq d-1$. 
Let $\kappa:V(G) \rightarrow [m]$ be a map satisfying $\kappa(u)=\kappa(v)$, and let $\kappa'$ be the restriction of $\kappa$ to $V(G/uv)$.
For $U \subseteq V(G)$, define $t_\kappa(U):=(|U \cap \kappa^{-1}(i)|)_i \in \mathbb{Z}_{\geq 0}^{m}$.
For $\bm{a} \in \mathbb{Z}_{> 0}^m$, suppose that there is a $(d-1)$-set $C \subseteq N_G(u) \cap N_G(v)$ such that $t_\kappa(C+u+v)=\bm{a}+\bm{e}_j$ for some $j \in [m]$, where $\bm{e}_j$ denotes the $j$th unit coordinate vector.
Then $G$ is $(\kappa,\bm{a})$-sparse rigid if $G/uv$ is $(\kappa',\bm{a})$-sparse rigid.
\end{corollary}
\begin{proof}
One can easily check that in the setting of $(\kappa,\bm{a})$-sparse rigidity, $(d+1)$-set $U$ satisfies the condition (H) if and only if $t_\kappa(U)=\bm{a}+\bm{e}_j$ for some $j \in [m]$. Now the statement follows from Lemma~\ref{lem:hall}.
\end{proof}

\begin{proof}[Proof of Theorem~\ref{thm:a-bal}]
We prove the statement by the induction on $|V(\Delta)|$.
When $\Delta$ is trivial, $G(\Delta)=K_d$, a complete graph on $d$ vertices. By Lemma~\ref{lem:gen_pos}, $p(V(G))$ spans $(d-1)$-dimensional affine space for any generic $(\kappa,\bm{a})$-sparse $p$.
Hence $G(\Delta)$ is $(\kappa,\bm{a})$-sparse rigid.

Suppose that $\Delta$ is nontrivial.
By $a_1 \geq 2$, we can pick $uv \in E(G)$ with $\kappa(u)=\kappa(v)=1$.
Let $\Delta_1^+,\ldots,\Delta_t^+$ be the minimal $(d-1)$-cycle complexes given in Lemma~\ref{lem:Fog_dec} with respect to $\Delta$ and $uv$.
For $U \subseteq V(G)$, define the type of $U \subseteq V(G)$ by $t_\kappa(U):=(|U \cap \kappa^{-1}(i)|)_i \in \mathbb{Z}_{\geq 0}^{m}$.
\begin{claim}\label{claim:a-bal}
For each $i$, there is a $(d-1)$-set $C \subseteq N_{G(\Delta_i^+)}(u) \cap N_{G(\Delta_i^+)}(v)$ such that $t_\kappa(C+u+v)=\bm{a}+\bm{e}_k$ for some $k \in [m]$.
\end{claim}
\begin{proof}[Proof of claim]
Let $F^*$ be a facet of $\Delta_i^+$ containing $u$ and $v$, which exists by Lemma~\ref{lem:Fog_dec} (a). 
By Lemma~\ref{lem:basic} (ii), for every $w \in F^* - u -v$, there exists $x (\neq w)$ such that $F^*-w+x$ is a facet of $\Delta_i^+$. 
Then $C:=F^*-u-v+x$ is included in the common neighborhood of $u$ and $v$ in $G(\Delta_i^+)$.
We show that for an appropriate choice of $w$, $t_\kappa(C+u+v)=\bm{a}+\bm{e}_k$ holds for some $k \in [m]$.

By Lemma~\ref{lem:Fog_dec} (b), for every facet $F$ of $\Delta_i^+$ containing $u$ and $v$, $F-v$ is a $(d-2)$-face of $\Delta$, and thus $t_\kappa(F-v)=\bm{a} -\bm{e}_j$ for some $j \in [m]$.
Hence, we have the following property ($\star$):
\begin{quote} 
($\star$) for every facet $F$ of $\Delta_i^+$ containing $u$ and $v$, $t_\kappa(F) = \bm{a}-\bm{e}_j+\bm{e}_1$ for some $j \in [m]$.
\end{quote}
By property ($\star$) of $F^*$, $t_\kappa(F^*) = \bm{a}-\bm{e}_j+\bm{e}_1$ for some $j \in [m]$.
If $j=1$, for any choice of $w$, we have $t_\kappa(C+u+v)=\bm{a}+\bm{e}_{\kappa(x)}$ as desired.
If $j \neq 1$, pick $w$ from $F^* \cap \kappa^{-1}(j)$, which is not empty by $a_j \geq 2$. Since $F^*-w+x$ also satisfies ($\star$), it follows that $\kappa(x)=j$. Hence we have $t_\kappa(C+u+v)=\bm{a}+\bm{e}_1$ as desired.
\end{proof}

Let $\kappa_i$ (resp. $\kappa_i'$) be the restriction of $\kappa$ to $V(\Delta_i^+)$ (resp. $V(\Delta_i^+/uv)$).
Then $\Delta_i^+/uv$ is $\bm{a}$-balanced and $\kappa_i'$ is an $\bm{a}$-coloring of $\Delta_i^+/uv$.
$\Delta_i^+/uv$ is also a minimal $(d-1)$-cycle complex by Lemma~\ref{lem:Fog_dec} (c).
Thus by induction hypothesis, $G(\Delta_i^+/uv)$ is $(\kappa_i',\bm{a})$-sparse rigid.
By Corollary~\ref{cor:hall} and Claim~\ref{claim:a-bal}, $G(\Delta_i^+)$ is $(\kappa_i,\bm{a})$-sparse rigid for each $i$.

By Lemma~\ref{lem:gen_pos} and the property ($\star$), for a facet $F$ of $\Delta_i^+$ and a generic $(\kappa,\bm{a})$-sparse point configuration $p$, $p(F)$ spans a $(d-1)$-dimensional affine subspace.
Hence by Lemma~\ref{lem:Fog_dec} (d) and (e), we can deduce that $G(\Delta)=\bigcup_{i=1}^t G(\Delta_i^+)$ is $(\kappa,\bm{a})$-sparse rigid by the repeated application of Lemma~\ref{lem:gluing}.
\end{proof}

\section{$(\kappa,\bm{a})$-sparse rigidity of homology manifolds} \label{sec:7}
In this section, we prove that Conjecture~\ref{conj:main} holds for any $\bm{a}\neq(d-1),(1,d-1)$ if $d \geq 4$ and $\Delta$ is a homology $(d-1)$-manifolds.

Conjecture~\ref{conj:main} was verified for balanced simplicial $2$-spheres by Cook et al.~\cite{3-polytope}.
\begin{theorem}[Cook et al.~\cite{3-polytope}] \label{thm:3-poly}
For a balanced simplicial $2$-sphere $\Delta$ with a proper $3$-coloring $\kappa$, $G(\Delta)$ is $(\kappa,\bm{a})$-sparse rigid.
\end{theorem}
Kalai~\cite{Kal} defined a class $\mathcal{C}_d$ of $(d-1)$-dimensional pseudomanifolds for $d\geq 3$ as follows: $\mathcal{C}_3$ is the class of simplicial $2$-sphere, and for $d\geq 4$, a $(d-1)$-dimensional pseudomanifold $\Delta$ belongs to $\mathcal{C}_d$ if $\lk_\Delta(v) \in \mathcal{C}_{d-1}$ for every $v \in V(\Delta)$.
For $d\geq 4$, $\mathcal{C}_d$ includes all homology $(d-1)$-manifolds (over any field).
Kalai~\cite{Kal} showed that if $\Delta \in \mathcal{C}_d$ ($d \geq 3$), $G(\Delta)$ is rigid in $\mathbb{R}^d$.
In the case of $(\kappa,\bm{a})$-sparse rigidity, we have the following theorem.
\begin{theorem} \label{thm:C_d}
For $d \geq 3$, let $\bm{a}=(a_1,\ldots,a_m) \in \mathbb{Z}_{>0}^m$ be a positive integer vector with $\sum_{i=1}^m a_i=d$.
Let $\Delta$ be an $\bm{a}$-balanced pseudomanifold satisfying $\Delta \in \mathcal{C}_d$, and let $\kappa$ be an $\bm{a}$-coloring of $\Delta$.
If $\bm{a}\neq (d-1,1), (1,d-1)$, $G(\Delta)$ is $(\kappa,\bm{a})$-sparse rigid.
\end{theorem}
In the proof of Theorem~\ref{thm:C_d}, we use cone lemma for $(\kappa,\bm{a})$-sparse rigidity, which we first prove in the generality of $L$-sparse rigidity.
\begin{lemma} \label{lem:list-cone}
Let $G$ be a graph and $G*\{v\}$ be its cone graph.
Let $L:V(G) \cup \{v\} \rightarrow 2^{[d+1]}$ be a map with $L(v) \neq \emptyset$.
Suppose that there is $i \in L(v)$ such that, for each $u \in V(G)$, either $i \not\in L(u)$ or $|L(v)\setminus L(u)| \leq 1$ holds.
Define $L':V(G) \rightarrow 2^{[d+1]\setminus\{i\}}$ by $L'(u):=L(u) \cup L(v) \setminus \{i\}$ if $i \in L(u)$ and $L'(u):=L(u)$ otherwise, and identify $2^{[d+1]\setminus\{i\}}$ with $2^{[d]}$.
Then $G$ is $L'$-sparse rigid if and only if $G * \{v\}$ is $L$-sparse rigid.
\end{lemma}
\begin{proof}
Identify $H:=\{x \in \mathbb{R}^{d+1}:x_i=0\}$ with $\mathbb{R}^d$.
Suppose that $G * \{v\}$ is $L$-sparse rigid.
Let $p:V(G) \cup \{v\} \rightarrow \mathbb{R}^{d+1}$ be a generic $L$-sparse configuration. 
Then $(G * \{v\},p)$ is infinitesimally rigid in $\mathbb{R}^{d+1}$. 
As $i \in L(v)$, we have $p(v) \not\in H$, and $p(u)-p(v)$ is not parallel to $H$ for any $u \in V(G)$.
For each $u \in V(G)$, let $p_H(u)$ be the intersection of $H$ and the line passing through $p(u)$ and $p(v)$. 
By Lemma~\ref{lem:coning}, $(G,p_H)$ is infinitesimally rigid in $\mathbb{R}^d$.
By the definition of $L'$, $p_H$ is $L'$-sparse.
Thus $G$ is $L'$-sparse rigid.

To see the other direction, suppose that $G$ is $L'$-sparse rigid.
There is $q:V(G) \rightarrow \mathbb{R}^d \cong H$ such that $(G,q)$ is infinitesimally rigid in $\mathbb{R}^d$ and $q$ is $L'$-sparse and $\supp q(u)=L'(u)$ for all $u \in V(G)$.
We define $p:V  \cup \{v\} \rightarrow \mathbb{R}^{d+1}$ as follows.
Let $p(v) \in \mathbb{R}^{d+1}$ be a point such that $\supp(p(v))=L(v)$ and $p(v)_j \neq q(u)_j$ for any $j \in L(v)$ and $u \in V(G)$.
For $u \in V(G)$, if $i \not \in L(u)$ or $L(u) \subseteq L(v)$, let $p(u):=q(u)$. 
For $u \in V(G)$, if $i \in L(u)$ and $L(u) \not\subseteq L(v)$, by the assumption on $L$, we must have $L(v) \setminus L(u)=\{j\}$ for some $j (\neq i)$. For such $u$, as $q(u)_j\neq0,p(v)_j$, there is a unique real number $t (\neq 0,1) \in \mathbb{R}$ such that $(tq(u)+(1-t)p(v))_j=0$, so let $p(u):=tq(u)+(1-t)p(v)$.
One can see that $p$ is $L$-sparse.
By Lemma~\ref{lem:coning}$, (G*\{v\},p)$ is infinitesimally rigid in $\mathbb{R}^{d+1}$.
Hence $G*\{v\}$ is $L$-sparse rigid.
\end{proof}
We obtain the following corollary of Lemma~\ref{lem:list-cone} for $(\kappa,\bm{a})$-sparse rigidity.
\begin{corollary} \label{cor:cone-sparse}
Let $G$ be a graph, $\kappa:V(G) \rightarrow [m]$ a map, and $\bm{a} \in \mathbb{Z}_{>0}^m$ a positive integer vector. 
Let $G*\{v\}$ be the cone graph of $G$.
Suppose that $G$ is $(\kappa,\bm{a})$-sparse rigid. We have the followings: 
\begin{itemize}
\item[(i)] For the extension $\kappa':V(G) \cup \{v\} \rightarrow [m+1]$ of $\kappa$ defined by $\kappa'(v)=m+1$, $G*\{v\}$ is $(\kappa',(\bm{a},1))$-sparse rigid.
\item[(ii)] For the extension $\kappa':V(G) \cup \{v\} \rightarrow [m]$ of $\kappa$ defined by $\kappa'(v)=i$ for some $i \in [m]$, $G*\{v\}$ is $(\kappa',\bm{a}+\bm{e}_i)$-sparse rigid.
\end{itemize}
\end{corollary}

\begin{proof}[Proof of Theorem~\ref{thm:C_d}]
If $m=1$, the statement is the result of Kalai~\cite{Kal}. 
If $m=2$, the statement follows from Theorem~\ref{thm:a-bal}.
We prove the statement for $m \geq 3$ by the induction on $d$.
If $d=3$, we have $\bm{a}=(1,1,1)$, and the statement follows from Theorem~\ref{thm:3-poly}.
Suppose that $d \geq 4$ and $\Delta \in \mathcal{C}_d$ is $\bm{a}$-balanced. Let $\kappa$ be an $\bm{a}$-coloring of $\Delta$.
Define $U \subseteq V(G)$ as follows:
if $a_j=1$ for all $j$, let $U:=\kappa^{-1}(\{1,2\})$, and
if $a_j \geq 2$ for some $j \in [m]$, let $U:=\kappa^{-1}(j)$.

\begin{claim} \label{claim:star}
For each $v \in U$, $G(\st_\Delta(v))$ is $(\kappa|_{V(\st_\Delta(v))},\bm{a})$-sparse rigid.
\end{claim}
\begin{proof}[Proof of claim]
First consider the case in which $a_j=1$ for all $j$. In this case $m=d \geq 4$.
Let $\kappa':V(\lk_\Delta(v)) \rightarrow [m-1]$ be the restriction of $\kappa$ to $V(\lk_\Delta(v))$, where $[m-1]$ is identified with $[m]\setminus\{\kappa(v)\}$.
Then $\kappa'$ is a $\bm{b}$-coloring of $\lk_\Delta(v)$, where $\bm{b}=(1,\ldots,1) \in \mathbb{Z}^{m-1}$.
As $\lk_\Delta(v) \in \mathcal{C}_{d-1}$, by the induction hypothesis, $G(\lk_\Delta(v))$ is $(\kappa',\bm{b})$-sparse rigid.
Thus by Corollary~\ref{cor:cone-sparse} (i), the claim follows.

Now consider the case in which $a_j\geq2$ for some $j$.
Then $\kappa|_{V(\lk_\Delta(v))}$ is an $(\bm{a}-\bm{e}_j)$-coloring of $\lk_\Delta(v)$, so $G(\lk_\Delta(v))$ is $(\kappa|_{V(\lk_\Delta(v))},\bm{b})$-sparse rigid by the induction hypothesis.
Thus by Corollary~\ref{cor:cone-sparse} (ii), the claim follows.
\end{proof}

By definition of $U$, $|U \cap F| \geq 2$ for every facet $F$ of $\Delta$.
As $\Delta$ is a pseudomanifold, $\Delta$ is strongly connected. Hence $G(\Delta)[U]$ is connected.
Thus, the vertices of $U$ can be ordered as $v_1,\ldots,v_t$ in such a way that $\bigcup_{j<i} \st_\Delta(v_j)$ and $\st_\Delta(v_i)$ share at least one facet for each $i\geq 2$.
By Lemma~\ref{lem:gen_pos}, for any generic $(\kappa,\bm{a})$-sparse configuration $p:V(\Delta)\rightarrow \mathbb{R}^d$ and any facet $F$ of $\Delta$, $p(F)$ spans a $(d-1)$-dimensional affine subspace.
Hence, by Lemma~\ref{lem:gluing} and Claim~\ref{claim:star}, $G(\Delta)=\bigcup_{v \in U} G(\st_\Delta(v))$ is $(\kappa,\bm{a})$-sparse rigid.
\end{proof}

The assumption $\bm{a}\neq(d-1),(1,d-1)$ in Theorem~\ref{thm:C_d} is necessary.
Cook et al.~\cite{3-polytope} showed that there is a $(2,1)$-balanced simplicial $2$-sphere $\Delta$ and its $(2,1)$-coloring $\kappa$ such that $G(\Delta)$ is not $(\kappa,\bm{a})$-sparse rigid.
Based on the construction given in \cite{3-polytope}, for each $d\geq 3$, we give a $(d-1,1)$-balanced simplicial $(d-1)$-sphere $\Delta$ and a $(d-1,1)$-coloring $\kappa$ of $\Delta$ such that $G(\Delta)$ is not $(\kappa,(d-1,1))$-sparse rigid.
Moreover, $\dim \ker R(G(\Delta),p)-\binom{d+1}{2}$ can be arbitrarily large for any $(\kappa,(d-1,1))$-sparse point configuration $p$.
\begin{example} \label{eg:(d-1,1)}
Let $\Gamma$ be a stacked $(d-1)$-simplicial sphere.
Let $\Delta$ be the simplicial $(d-1)$-sphere obtained from $\Gamma$ by subdividing all facets of $\Gamma$.
Then $\Delta$ is a stacked $(d-1)$-sphere with $f_0(\Delta)+f_{d-1}(\Delta)$ vertices.
Define $\kappa:V(\Delta) \rightarrow \{1,2\}$ by $\kappa(v)=1$ if $v \in V(\Gamma)$ and $\kappa(v)=2$ otherwise. Then $\kappa$ is a $(d-1,1)$-coloring of $\Delta$.
Let $p:V(\Delta) \rightarrow \mathbb{R}^d$ be a $(\kappa,(d-1,1))$-sparse point configuration of $G(\Delta)$.
As $\Delta$ is a stacked sphere, we have $\dim \ker R(G(\Delta),p)^\top = \dim \ker R(G(\Delta),p)-\binom{d+1}{2}$. 
The subframework $(G(\Gamma),p|_{V(\Gamma)})$ of $(G(\Delta),p)$ is in a $(d-1)$-dimensional affine subspace and has $d f_0(\Gamma)-\binom{d+1}{2}$ edges. Thus the dimension of the linear space of the equilibrium stresses of $(G(\Gamma),p|_{V(\Gamma)})$ is at least $d f_0(\Gamma)-\binom{d+1}{2} - ((d-1)f_0(\Gamma)-\binom{d}{2})=f_0(\Gamma)-d$.
Hence $\dim \ker R(G(\Delta),p)-\binom{d+1}{2}$ can be arbitrarily large by setting $f_0(\Gamma)$ large enough.
\end{example}

\section{Further observations on $L$-sparse rigidity}
We can prove further results similar to Theorem~\ref{thm:a-bal}.
The following proposition treats the case in which $\Delta$ satisfies the weaker condition than $\bm{a}$-balancedness.
\begin{prop}
For $m \geq 1$, let $\bm{a}=(a_1,\ldots,a_m) \in \mathbb{Z}_{> 0}^m$ be a positive integer vector with $a_1 \geq 4$ and $a_i \geq 2$ for all $i \in [m]$. 
Let $\Delta$ be a minimal $(d-1)$-cycle complex, where $d=\sum_{i=1}^m a_i$.
Let $\kappa:V(\Delta) \rightarrow [m]$ be a map such that, for any facet $F$ of $\Delta$, $\left(|F \cap \kappa^{-1}(i)|\right)_i=\bm{a}+\bm{e}_j-\bm{e}_1 \in \mathbb{Z}_{> 0}^m$ holds for some $j \in [m]$.
Then, $G(\Delta)$ is $(\kappa,\bm{a})$-sparse rigid.
\end{prop}
\begin{proof}
The proof is similar to that of Theorem~\ref{thm:a-bal}.
Define a type $t_\kappa(U)$ of $U \subseteq V(\Delta_i^+)$ as $(|U \cap \kappa^{-1}(i)|)_i \in \mathbb{Z}_{\geq 0}^m$.
Note that by Lemma~\ref{lem:gen_pos}, for a generic $(\kappa,\bm{a})$-sparse point configuration $p:V(\Delta)\rightarrow \mathbb{R}^d$ and $U \subseteq V(\Delta)$ with $|U|=d+1$, $p(U)$ is affinely independent if and only if $t_\kappa(U)=\bm{a}+\bm{e}_j$ for some $j \in [m]$.

The proof is done by the induction on $|V(\Delta)|$. As a base case, if $\Delta$ is a trivial minimal $(d-1)$-cycle complex, $G(\Delta)=K_d$ and by Lemma~\ref{lem:gen_pos}, for a generic $(\kappa,\bm{a})$-sparse point configuration $p:V(\Delta)\rightarrow \mathbb{R}^d$, $p(V(\Delta))$ spans a $(d-1)$-dimensional affine space. 
Hence $G(\Delta)$ is $(\kappa,\bm{a})$-sparse rigid.

Suppose that $\Delta$ is a nontrivial minimal $(d-1)$-cycle complex. Pick $u,v \in \kappa^{-1}(1)$ with $uv \in E(G)$, which exist by $a_1 \geq2$. Let $\Delta_1^+,\ldots,\Delta_t^+$ be the nontrivial minimal $(d-1)$-cycle complexes given in Lemma~\ref{lem:Fog_dec} with respect to $\Delta$ and $uv$.
\begin{claim} \label{claim:nonbal}
For each $i$, there is a $(d-1)$-set $C \subseteq N_{G(\Delta_i^+)}(u) \cap N_{G(\Delta_i^+)}(v)$ such that $t_\kappa(C+u+v)=\bm{a}+\bm{e}_k$ for some $k \in [m]$.
\end{claim}
\begin{proof}[Proof of claim]
Pick a facet $F^* \in \Delta_i^+$ containing $u$ and $v$. 
For every $w \in F^*-u-v$, there is $x (\neq w)$ such that $F^*-w+x$ is a facet of $\Delta_i^+$.
We show that for an appropriate choice of $w \in F^*-u-v$, $C:=F^*+x-u-v$ satisfies $t_\kappa(C+u+v)=\bm{a}+\bm{e}_j$ for some $j \in [m]$.

By Lemma~\ref{lem:Fog_dec} (b), for every facet $F$ of $\Delta_i^+$ containing $v$, $t_\kappa(F-v)=\bm{b}-\bm{e}_j$ for some $j \in [m]$ and some $\bm{b} \in \{\bm{a}+\bm{e}_k-\bm{e}_1: k \in [m]\}$. 
Hence, we have the following:
\begin{quote}
($\star$) for every facet $F$ of $\Delta_i^+$ containing $v$, $t_\kappa(F)=\bm{a}+\bm{e}_k-\bm{e}_j$ for some $k,j \in [m]$.
\end{quote}
By the property ($\star$) of $F^*$, $t_\kappa(F^*)=\bm{a}+\bm{e}_k-\bm{e}_j$ for some $k,j \in [m]$.
If $k=j$, for any choice of $w$, we have $t_\kappa(C+u+v)=\bm{a}+\bm{e}_{\kappa(x)}$ as desired.
Thus assume that $k \neq j$.
If $j=1$, pick $w \in F^*-u-v$ with $\kappa(w)=1$. Note that such $w$ exists by $a_1 \geq 4$.
Then for $F^*-w+x$ to satisfy $(\star)$, $\kappa(x)$ must be $1$, and we get $t_\kappa(C+u+v)=\bm{a}+\bm{e}_k$ as desired.
If $j \neq 1$, by $a_j \geq 2$, we can pick $w \in F-u-v$ with $\kappa(w)=j$. Then from the property $(\star)$ of $F^*-w+x$, we have $\kappa(x)=j$. Thus we get $t_\kappa(C+u+v)=\bm{a}+\bm{e}_k$ as desired.
\end{proof}

The restriction of $\kappa$ to $V(\Delta_i^+)$ (resp. $V(\Delta_i^+/uv)$) satisfies the assumption for $\Delta_i^+$ (resp. $\Delta_i^+/uv$).
Hence by the same argument as the proof of Theorem~\ref{thm:a-bal}, from Claim~\ref{claim:nonbal}, we can deduce that $G(\Delta)$ is $(\kappa,\bm{a})$-sparse rigid by Lemma~\ref{lem:gluing}, Lemma~\ref{lem:vertex splitting}, Lemma~\ref{lem:Fog_dec} (d) and (e).
\end{proof}

The following proposition asserts that the assumption $a_i \geq 2$ for all $i$ in Theorem~\ref{thm:a-bal} can be weakened to $a_1 \geq 2$ (and no assumption on $a_i$ for $i \neq 1$) if vertices in $\kappa^{-1}(1)$ are allowed to have full support.
\begin{prop}
Let $d > a \geq 2$ be integers and let $\bm{b} \in \mathbb{Z}_{>0}^m$ be an integer vector with $d-a=\sum_{i=1}^m b_i$.
Let $\Delta$ be a minimal $(d-1)$-cycle complex and $X \subseteq V(\Delta)$ be a $(\geq a)$-transversal set of $\Delta$. 
Let $\kappa:V(\Delta) \setminus X \rightarrow [m]$ be a map satisfying $|F \cap \kappa^{-1}(i)| \leq b_i$ for every $F \in \Delta$ and $i \in [m]$.

Let $I_1,\ldots,I_m$ be disjoint subsets of $[d]$ with $|I_i|=b_i$ for $i \in [m]$ and define a map $L^*:V(\Delta) \rightarrow 2^{[d]}$ by $L^*(v)=[d]$ if $v \in X$ and $L^*(v)=I_{\kappa(v)}$ if $v\in V(\Delta) \setminus X$.
Then $G(\Delta)$ is $L^*$-sparse rigid.
\end{prop}
\begin{proof}
For $U \subseteq V(\Delta)$, let $t(U):=(|U \cap X|,|U \cap \kappa^{-1}(1)|, \ldots,|U \cap \kappa^{-1}(m)|) \in \mathbb{Z}_{\geq 0}^{m+1}$.
The assumption on $X$ and $\kappa$ is that, for each facet $F$ of $\Delta$, $t(F)=(a+\sum_{i=1}^m d_i,\bm{b}-\bm{d})$ for some $\bm{d} \in \mathbb{Z}_{\geq 0}^m$ with $\bm{d} \leq \bm{b}$.
One can check that a $(d+1)$-set $U \subseteq V(\Delta)$ satisfies the condition (H) for $L^*$ if $t(U)=(a+\sum_{i=1}^m d_i,\bm{b}-\bm{d})+\bm{e}_j$ for some $j \in [m+1]$ and some $\bm{d} \in \mathbb{Z}_{\geq 0}^m$ with $\bm{d} \leq \bm{b}$.

The proof is done by the induction on $|V(\Delta)|$. When $\Delta$ is trivial, the statement easily follows.
Suppose that $\Delta$ is a nontrivial minimal $(d-1)$-cycle complex. Pick $u,v \in X$ with $uv \in E(\Delta)$, which exist by $a \geq 2$.
Let $\Delta_1^+,\ldots,\Delta_t^+$ be the nontrivial minimal $(d-1)$-cycle complexes given in Lemma~\ref{lem:Fog_dec} with respect to $\Delta$ and $uv$.
\begin{claim}
For each $i$, there exists $C \in N_{G(\Delta_i^+)}(u) \cap N_{G(\Delta_i^+)}(v)$ such that $C+u+v$ satisfies the condition (H) for $L^*$.
\end{claim}
\begin{proof}[Proof of claim]
For each facet $F$ of $\Delta_i^+$ containing $v$, we have $t(F)=(a+\sum_{i=1}^m d_i,\bm{b}-\bm{d})$ for some $\bm{d} \in \mathbb{Z}_{\geq 0}^m$ with $\bm{d} \leq \bm{b}$.
Pick $F^* \in \Delta_i^+$ with $u,v  \in F^*$, and then pick $w \in F^*-u-v$.
Then there is $x (\neq w)$ such that $F^*+x-w$ is a facet of $\Delta_i^+$.
Now $C:=F^*+x-u-v$ is the desired set.
\end{proof}
By the similar argument to the proof of Theorem~\ref{thm:a-bal}, the $L^*$-sparse rigidity of $G(\Delta)$ is deduced.
\end{proof}

\paragraph{Acknowledgement}
I would like to thank my advisor Shin-ichi Tanigawa for helpful discussions and helpful comments on the presentation.

\bibliographystyle{plain}
\bibliography{myreference}

\end{document}